\documentclass[11pt]{article}

\usepackage[a4paper,margin=1in]{geometry}
\usepackage[T1]{fontenc}
\usepackage[utf8]{inputenc}
\usepackage{lmodern}
\usepackage{microtype}
\usepackage{amsmath,amssymb,amsthm,mathtools}
\usepackage{array,booktabs,longtable}
\usepackage{enumitem}
\usepackage{xcolor}
\usepackage{listings}
\usepackage{fancyvrb}
\usepackage[hidelinks]{hyperref}

\newtheorem{theorem}{Theorem}[section]
\newtheorem{proposition}[theorem]{Proposition}

\theoremstyle{definition}

\newtheorem{remark}[theorem]{Remark}

\newcommand{\F}{\mathbb F}
\newcommand{\Sym}{\operatorname{Sym}}
\newcommand{\GL}{\operatorname{GL}}
\newcommand{\Span}{\operatorname{span}}

\lstdefinestyle{pythonstyle}{
	language=Python,
	basicstyle=\ttfamily\scriptsize,
	breaklines=true,
	breakatwhitespace=false,
	showstringspaces=false,
	columns=fullflexible,
	frame=single,
	numbers=left,
	numberstyle=\tiny,
	tabsize=4,
	captionpos=b
}

\title{Counterexamples to Conjectures of Wehlau on Noether Numbers}

\author{Muhammad Fazeel Anwar\\
	\small Department of Mathematics, Sukkur IBA University, Sukkur, Pakistan. \\ fazeel.anwar@iba-suk.edu.pk}

\begin{document}

\maketitle

\begin{abstract}
		Let $G$ be a finite group and let $V$ be a finite-dimensional $G$-module over a field $k$. We construct explicit counterexamples in characteristic $2$ to several questions and conjectures of Wehlau concerning Noether numbers. For the $2$-group $G=D_8$, we exhibit a submodule $U\subseteq V$, with $\dim_kU=5$ and $\dim_kV=6$, such that $\beta\bigl(k[U]^G\bigr)=6>5=\beta\bigl(k[V]^G\bigr)$, thereby disproving submodule monotonicity. Writing $X=V^*$ and $Q=U^*$, the corresponding nonsplit exact sequence $0\longrightarrow k\longrightarrow X\longrightarrow Q\longrightarrow0$
		also satisfies $\beta\bigl(k[Q]^G\bigr)=8>6=\beta\bigl(k[Q^*]^G\bigr)$
		and $\beta\bigl(k[Q]^G\bigr)=8>5=\beta\bigl(k[X]^G\bigr)$. Thus the modular Noether number need not be invariant under duality, and quotient monotonicity also fails. Notably, the basic counterexamples already occur for $2$-groups in defining characteristic. The constructions remain valid over every field of characteristic $2$, and the $D_8$ conclusions propagate by inflation to every finite group admitting $D_8$ as a quotient.
	\end{abstract}

	\section{Introduction}
	Let $V$ be a finite-dimensional vector space over a field $k$ and let $G$ be a finite subgroup of $\GL(V)$. We choose a basis $\{z_0,\ldots,z_n\}$ for the dual space $V^*$. The action of $G$ on $V$ induces a natural action on $V^*$, which extends to an action on the symmetric algebra $k[V]=\Sym(V^*)=k[z_0,\ldots,z_n]$. We denote the ring of invariants by $k[V]^G$. A representation of $G$ is said to be modular if the characteristic of $k$ divides $|G|$. For a finitely generated graded algebra $R=\bigoplus_{i=0}^{\infty}R_i$, we denote by $R_{+}=\bigoplus_{i>0}R_i$ its ideal of positive-degree elements. An element of $R$ is called decomposable if it lies in $(R_{+})^2$. The Noether number of $R$, denoted by $\beta(R)$, is the largest degree of a homogeneous indecomposable element of $R$ \cite{ShWe}. When the field and group are fixed, we write $\beta(V)=\beta(k[V]^G)$.

	Emmy Noether \cite{Noe} proved that $\beta(V)\leq |G|$ when $\operatorname{char}(k)=0$ or $\operatorname{char}(k)>|G|$. Fleischmann \cite{Fl} and Fogarty \cite{F} subsequently proved the same bound whenever the characteristic of $k$ does not divide $|G|$. In the modular case, G\"obel \cite{G} proved that $\beta(V)\leq \binom{\dim V}{2}$ when $V$ is a permutation module. Hughes and Kemper \cite{HK} obtained degree bounds for arbitrary modular representations of a cyclic group of prime order.
	
	For a $G$-submodule $U$ of $V$, Wehlau \cite{OP} conjectured that $\beta(U)\leq\beta(V)$. Shank and Wehlau \cite{ShWe} proved the conjecture for the cyclic group of prime order over a field of the same characteristic. Chen and Zeng \cite{NP} recently gave a different proof. Zhang and Chen \cite{DG} proved the conjecture for dihedral groups of order $2p$ in characteristic $p>2$, and for dihedral groups of order $2m$ in characteristic $2$ when $m$ is odd. Wehlau \cite{OP} also asked whether quotient modules satisfy $\beta\bigl(k[V/U]^G\bigr)\leq\beta\bigl(k[V]^G\bigr)$.

	For $j\geq1$, the generalized Noether number is
	\[
	\beta_j(G,W)=
	\operatorname{topdeg}
	\left(
	\frac{\bigl(k[W]^G\bigr)_+}
	{\bigl(\bigl(k[W]^G\bigr)_+\bigr)^{j+1}}
	\right),
	\]
	so that $\beta_1(G,W)=\beta(k[W]^G)$ \cite{CD}.

	In this paper we take the $2$-group $D_8$ over $k=\mathbb F_2$ and construct a six-dimensional representation $V$ together with a five-dimensional subrepresentation $U$. The same construction also produces failures of duality invariance and quotient monotonicity. Our main results are the following.

	\begin{theorem}\label{thm:main}
		For $k=\F_2$ and $G=D_8$, there are $G$-modules $U\subseteq V$ with $\dim_kU=5$ and $\dim_kV=6$ such that $\beta\bigl(k[U]^G\bigr)=6>5=\beta\bigl(k[V]^G\bigr)$. Consequently, the submodule-monotonicity conjecture of Wehlau is false.
	\end{theorem}

	\begin{theorem}\label{thm:duality}
		For the modules in Theorem~\ref{thm:main}, put $X=V^*$ and $Q=U^*$. Then $\beta\bigl(k[X]^G\bigr)=\beta\bigl(k[X^*]^G\bigr)=5$, whereas $\beta\bigl(k[Q]^G\bigr)=8 \, \text{and}\, \beta\bigl(k[Q^*]^G\bigr)=6$. In particular, $\beta\bigl(k[Q]^G\bigr)=8>6=\beta\bigl(k[Q^*]^G\bigr)$, so the modular Noether number is not invariant under duality.
	\end{theorem}

	\begin{theorem}\label{thm:quotient}
		For the module $X$ in Theorem~\ref{thm:duality}, there is a one-dimensional trivial $G$-submodule $L$ such that $Q=X/L$ and $\beta\bigl(k[X/L]^G\bigr)=8>5=\beta\bigl(k[X]^G\bigr)$. Hence, quotient monotonicity is false.
	\end{theorem}

	\begin{theorem}\label{thm:generalized}
		For the pair $U\subseteq V$ in Theorem~\ref{thm:main}, the first four generalized Noether numbers are
		\[
		\begin{array}{c|cccc}
			j&1&2&3&4\\ \hline
			\beta_j(G,V)&5&9&13&17\\
			\beta_j(G,U)&6&10&14&18.
		\end{array}
		\]
		Thus $\beta_j(G,U)>\beta_j(G,V), \, (1\leq j\leq4)$.
	\end{theorem}

	\begin{theorem}\label{thm:inflation}
		The conclusions of Theorems~\ref{thm:main}, \ref{thm:duality}, \ref{thm:quotient}, and~\ref{thm:generalized} remain valid over every field of characteristic $2$. They also remain valid after inflation to every finite group admitting $D_8$ as a quotient.
	\end{theorem}

	The remainder of the paper proves these results and records their consequences.

	\section{Proof of Theorem~\ref{thm:main}}
	
	We take $k=\F_2$ and $X=V^*=\Span_k\{z_0,z_1,z_2,z_3,z_4,z_5\}$. Define $r,s\in\GL(X)$ by
	\begin{equation}\label{eq:X-action-table}
		\begin{array}{c|cccccc}
			&z_0&z_1&z_2&z_3&z_4&z_5\\ \hline
			r&z_0&z_1&z_1+z_2&z_2+z_3&z_0+z_4&z_4+z_5\\[1mm]
			s&z_0&z_1&z_1+z_2&z_1+z_3&z_4&z_1+z_4+z_5.
		\end{array}
	\end{equation}

	Therefore we get 
	\begin{equation}\label{eq:X-matrices}
		[r]_X=
		\begin{pmatrix}
			1&0&0&0&1&0\\
			0&1&1&0&0&0\\
			0&0&1&1&0&0\\
			0&0&0&1&0&0\\
			0&0&0&0&1&1\\
			0&0&0&0&0&1
		\end{pmatrix},
		\qquad
		[s]_X=
		\begin{pmatrix}
			1&0&0&0&0&0\\
			0&1&1&1&0&1\\
			0&0&1&0&0&0\\
			0&0&0&1&0&0\\
			0&0&0&0&1&1\\
			0&0&0&0&0&1
		\end{pmatrix}.
	\end{equation}
	
	A direct calculation gives $r^4=s^2=1$ and $srs=r^{-1}$. Moreover, $r$ has order $4$ and $s\notin\langle r\rangle$, so $r$ and $s$ generate a subgroup isomorphic to $D_8$. The line $kz_0$ is fixed by $G$. We take $Q=X/kz_0$ and choose the basis $x_0=\overline z_1, \quad x_1=\overline z_2, \quad x_2=\overline z_3, \quad x_3=\overline z_4, \quad x_4=\overline z_5$.
	The induced action is given by
	\begin{equation}\label{eq:Q-action-table}
		\begin{array}{c|ccccc}
			&x_0&x_1&x_2&x_3&x_4\\ \hline
			r&x_0&x_0+x_1&x_1+x_2&x_3&x_3+x_4\\[1mm]
			s&x_0&x_0+x_1&x_0+x_2&x_3&x_0+x_3+x_4.
		\end{array}
	\end{equation}
	Therefore we have
	\begin{equation}\label{eq:Q-matrices}
		[r]_Q=
		\begin{pmatrix}
			1&1&0&0&0\\
			0&1&1&0&0\\
			0&0&1&0&0\\ 
			0&0&0&1&1\\
			0&0&0&0&1
		\end{pmatrix},
		\qquad
		[s]_Q=
		\begin{pmatrix}
			1&1&1&0&1\\
			0&1&0&0&0\\
			0&0&1&0&0\\
			0&0&0&1&1\\
			0&0&0&0&1
		\end{pmatrix}.
	\end{equation}
	
	Set $V=X^*$ and $U=Q^*$. Dualizing the canonical quotient map $X\twoheadrightarrow Q$ gives an injective $G$-homomorphism $U=Q^*\hookrightarrow X^*=V$. Hence $U$ is a submodule of $V$. More precisely, if $v_0,\ldots,v_5$ is the basis of $X^*$ dual to
	$z_0,\ldots,z_5$, then $U=(kz_0)^\perp =\Span_k\{v_1,v_2,v_3,v_4,v_5\} \subseteq V$. We now use the algorithm of \cite{Kem} and Magma \cite{mag} to compute the generators of the invariant rings $k[V]^G=k[z_0,z_1,z_2,z_3,z_4,z_5]^G$ and $k[U]^G=k[x_0,x_1,x_2,x_3,x_4]^G$. 
	
	\begin{proposition}\label{prop1}
		Let $R=\mathbb F _2[z_0,z_1,z_2,z_3,z_4,z_5]^G$. Then $R$ is generated by the homogeneous invariants $\{f_1,\cdots,f_{11}\}$ with $\deg(f_1),\ldots,\deg(f_{11}) = 1,1,2,2,3,3,3,4,4,4,5$. The invariants are $f_1=z_0$, $\qquad f_2=z_1$, $\qquad f_3=z_0z_4+z_4^2$, $\qquad f_4=z_1z_2+z_2^2$,  $\qquad f_5= z_0^2z_2+z_0^2z_3+z_0z_1z_2+z_0z_3^2 +z_1^2z_5+z_1z_5^2+z_2^2z_4+z_2z_4^2$, $\qquad f_6= z_0z_1z_3+z_0z_3^2+z_1z_2z_4+z_2^2z_4$, $\qquad f_7={} z_0z_1z_5+z_0z_2z_4+z_0z_3^2 +z_1^2z_4+z_1^2z_5+z_1z_5^2 +z_2^2z_4+z_2z_4^2$, $\qquad f_8= z_0^2z_1z_2 +z_0^2z_3z_4 +z_0^2z_4z_5 +z_0^2z_5^2 +z_0z_1z_2z_3 +z_0z_1z_3^2 +z_0z_1z_4^2 +z_0z_1z_4z_5 +z_0z_2^3 +z_0z_2^2z_3 +z_0z_2z_4^2 +z_0z_3z_4^2 +z_0z_4^2z_5 +z_0z_4z_5^2 +z_1^3z_3 +z_1^3z_5 +z_1^2z_2z_5 +z_1^2z_3^2 +z_1^2z_4^2 +z_1z_2^3 +z_1z_2^2z_5 +z_1z_4^3 +z_1z_4^2z_5 +z_2^4 +z_2^3z_4 +z_2z_4^3 +z_4^2z_5^2 +z_5^4$, $\qquad f_9= z_0z_1z_2^2 +z_0z_1z_2z_3 +z_0z_2^3 +z_0z_2^2z_3 +z_1^3z_3 +z_1^2z_2^2 +z_1^2z_2z_5 +z_1^2z_3^2 +z_1z_2^3 +z_1z_2^2z_4 +z_1z_2^2z_5 +z_2^3z_4$, $\qquad f_{10}= z_1^2z_2z_3 +z_1^2z_3^2 +z_1z_2^2z_3 +z_1z_2z_3^2 +z_2^2z_3^2 +z_3^4$, $\qquad f_{11}= z_0^3z_1z_2 +z_0^3z_3^2 +z_0^2z_1^2z_5 +z_0^2z_1z_2^2 +z_0^2z_2^3 +z_0^2z_2^2z_3 +z_0^2z_2^2z_4 +z_0^2z_2z_3^2 +z_0^2z_3^3 +z_0z_1^3z_2 +z_0z_1^3z_4 +z_0z_1^3z_5 +z_0z_1^2z_3z_5 +z_0z_1z_2^2z_3 +z_0z_1z_3^2z_5 +z_0z_2^3z_4 +z_0z_2^2z_3^2 +z_0z_2^2z_3z_4 +z_0z_2z_3^2z_4 +z_1^4z_3 +z_1^4z_4 +z_1^3z_2^2 +z_1^3z_2z_4 +z_1^3z_3^2 +z_1^3z_3z_4 +z_1^2z_2^3 +z_1^2z_2^2z_4 +z_1^2z_2z_4z_5 +z_1^2z_2z_5^2 +z_1^2z_3^2z_4 +z_1z_2^3z_4 +z_1z_2^2z_4z_5 +z_1z_2^2z_5^2 +z_2^4z_4$.
	\end{proposition}

	\begin{proof}
		Let $A=\mathbb F _2[f_1,\ldots,f_{11}] \subseteq R=\mathbb F_2[z_0,\ldots,z_5]^G$, as $r(f_i)=f_i=s(f_i) \qquad (1\leq i\leq11)$. For each \(d\geq0\), the homogeneous invariant space \(R_d\) was computed exactly over \(\mathbb F _2\) as $R_d=\ker(r-1)\cap\ker(s-1) \subseteq \mathbb F _2[z_0,\ldots,z_5]_d$. The resulting dimensions are
		\[
		\begin{array}{c|rrrrrrrrrrrrr}
			d&0&1&2&3&4&5&6&7&8&9&10&11&12\\ \hline
			\dim_{\mathbb F _2}R_d
			&1&2&5&11&23&41&71&117&186&281&414&594&834.
		\end{array}
		\]
		On the other hand, direct multiplication of the displayed invariants
		gives
		\[
		\dim_{\mathbb F _2}A_d
		=
		\dim_{\mathbb F _2}R_d
		\qquad (0\leq d\leq12).
		\]
		Since \(A_d\subseteq R_d\), it follows that $A_d=R_d \qquad (0\leq d\leq12)$. Consider the invariants $z_0$, $z_1$, $z_4(z_0+z_4)$, $z_2(z_1+z_2)$, $z_3(z_1+z_3)(z_2+z_3)(z_1+z_2+z_3)$, and
		\[
		\begin{aligned}
		N_G(z_5)
		&=\prod_{g\in G}g(z_5)\\
		&=z_5(z_0+z_5)(z_1+z_5)(z_0+z_1+z_5)\\
		&\quad\cdot(z_4+z_5)(z_0+z_4+z_5)
		(z_1+z_4+z_5)(z_0+z_1+z_4+z_5).
		\end{aligned}
		\]
		Their degrees are $1,1,2,2,4,8$. Their common zero over $\overline{\mathbb F}_2$ is the origin: successively, the first two invariants give $z_0=z_1=0$, the next two give $z_4=z_2=0$, the fifth gives $z_3=0$, and the norm then gives $z_5=0$. Hence they form a homogeneous system of parameters for $R$. Let $P$ denote the polynomial algebra generated by these parameters. By Symonds' bound \cite{Sy}, $R$ is generated as a $P$-module by homogeneous elements of degree at most
		\[
		(1-1)+(1-1)+(2-1)+(2-1)+(4-1)+(8-1)=12.
		\]
		Since $A_d=R_d$ for every $d\leq12$, every required $P$-module generator belongs to $A$. As $P\subseteq A$, it follows that $R\subseteq A$, and therefore $R=A=\mathbb F _2[f_1,\ldots,f_{11}]$. Finally, letting $R_+=\bigoplus_{d>0}R_d$, we obtain
		\[
		\dim_{\mathbb F _2}
		\left(R_d/(R_+^2)_d\right)
		=
		\begin{cases}
			2,&d=1,\\
			2,&d=2,\\
			3,&d=3,\\
			3,&d=4,\\
			1,&d=5,\\
			0,&6\leq d\leq12.
		\end{cases}
		\]
		The images of
		\[
		f_1,f_2;\qquad
		f_3,f_4;\qquad
		f_5,f_6,f_7;\qquad
		f_8,f_9,f_{10};\qquad
		f_{11}
		\]
		form bases of the corresponding nonzero indecomposable quotients.
		Thus none of the \(f_i\) can be omitted, and
		\(\{f_1,\ldots,f_{11}\}\) is a minimal homogeneous generating set of
		\(R\).
	\end{proof}

	\begin{proposition}\label{prop2}
		Let $S=\mathbb F _2[x_0,x_1,x_2,x_3,x_4]^G$. Then \(S\) is generated by the homogeneous invariants \(\{f_1,\ldots,f_9\}\), with $\deg(f_1),\ldots,\deg(f_9)	= 1,1,2,3,4,4,4,4,6$. The invariants are $ f_1=x_0, \qquad f_2=x_3, \qquad f_3=x_0x_1+x_1^2, \qquad f_4= x_0^2x_4+x_0x_1x_3+x_0x_4^2+x_1x_3^2, \qquad f_5= x_0^3x_2+x_0^2x_1^2+x_0^2x_1x_4+x_0^2x_2^2+x_0x_1^3+x_0x_1^2x_3+x_0x_1^2x_4+x_1^3x_3, \qquad f_6= x_0^3x_4+x_0x_1^2x_3+x_0x_3^2x_4+x_1x_3^3+x_3^2x_4^2+x_4^4, \qquad f_7= x_0^2x_1x_2+x_0^2x_2^2+x_0x_1^2x_2+x_0x_1x_2^2+x_1^2x_2^2+x_2^4, \qquad f_8=x_0^2x_1x_4+x_0^2x_2x_3+x_0x_1^2x_4+x_0x_1x_3x_4+x_0x_1x_4^2+x_0x_2^2x_3+x_1^2x_3x_4+x_1^2x_4^2, \qquad f_9=x_0^4x_1x_4+x_0^4x_2x_4+x_0^3x_1x_2x_3+x_0^3x_1x_3x_4+x_0^3x_2^2x_4+x_0^3x_2x_4^2+x_0^3x_3^2x_4+x_0^2x_1^3x_4+x_0^2x_1^2x_3x_4+x_0^2x_1^2x_4^2+x_0^2x_1x_2^2x_3+x_0^2x_1x_2x_3^2+x_0^2x_1x_3^3+x_0^2x_1x_3x_4^2+x_0^2x_1x_4^3+x_0^2x_2^2x_4^2+x_0^2x_3^2x_4^2+x_0x_1^3x_4^2+x_0x_1^2x_3^2x_4+x_0x_1^2x_4^3+x_0x_1x_2^2x_3^2+x_0x_1x_3^4+x_1^3x_3^2x_4+x_1^3x_3x_4^2$.

	\end{proposition}
	
	\begin{proof}
	
Put $T=\mathbb F_{2}[x_{0},x_{1},x_{2},x_{3},x_{4}], \qquad S=T^{G}, \qquad A=\mathbb F_{2}[f_{1},\ldots,f_{9}]\subseteq S$. The action of the generators of \(G=D_{8}\) on the degree-one component \(T_{1}=\Span_{\mathbb F_2}\{x_0,x_1,x_2,x_3,x_4\}\) is
			\[
			\begin{aligned}
				&r(x_{0})=x_{0},\qquad
				r(x_{1})=x_{0}+x_{1},\qquad
				r(x_{2})=x_{1}+x_{2},\\
				&r(x_{3})=x_{3},\qquad
				r(x_{4})=x_{3}+x_{4},
			\end{aligned}
			\]
			and
			\[
			\begin{aligned}
				&s(x_{0})=x_{0},\qquad
				s(x_{1})=x_{0}+x_{1},\qquad
				s(x_{2})=x_{0}+x_{2},\\
				&s(x_{3})=x_{3},\qquad
				s(x_{4})=x_{0}+x_{3}+x_{4}.
			\end{aligned}
			\]
Direct substitution gives $r(f_i)=f_i=s(f_i) \qquad (1\leq i\leq 9)$, so \(A\subseteq S\). We first exhibit a homogeneous system of parameters for \(S\). Set
			\[
			p_{1}=f_{1}=x_{0},\qquad
			p_{2}=f_{2}=x_{3},\qquad
			p_{3}=f_{3}=x_{0}x_{1}+x_{1}^{2},
			\qquad
			p_{4}=f_{6},\qquad
			p_{5}=f_{7}.
			\]
			Their degrees are $1,1,2,4,4$. We claim that their only common zero over \(\overline{\mathbb F}_{2}\) is the origin. Indeed, if $p_{1}=p_{2}=p_{3}=p_{4}=p_{5}=0$, then \(p_{1}=p_{2}=0\) gives $x_{0}=x_{3}=0$. It follows from $p_{3}=x_{0}x_{1}+x_{1}^{2}$ that $x_{1}^{2}=0$, and hence \(x_{1}=0\). Under the substitutions
			\(x_{0}=x_{1}=x_{3}=0\), the two remaining parameters reduce to $p_{4}=f_{6}=x_{4}^{4}, \qquad p_{5}=f_{7}=x_{2}^{4}$. Thus \(x_{4}=x_{2}=0\), proving the claim. Consequently,
			\[
			T/(p_{1},p_{2},p_{3},p_{4},p_{5})T
			\]
			is finite-dimensional over \(\mathbb F_{2}\). Hence \(T\) is module-finite over $P=\mathbb F_{2}[p_{1},p_{2},p_{3},p_{4},p_{5}]$. Since \(P\subseteq S\subseteq T\) and \(P\) is Noetherian, \(S\) is also
			module-finite over \(P\). Therefore \(p_{1},\ldots,p_{5}\) form a homogeneous system of parameters for \(S\). By Symonds' bound for secondary invariants, \(S\) is generated as a
			\(P\)-module by homogeneous elements of degree at most
			\[
			(1-1)+(1-1)+(2-1)+(4-1)+(4-1)=7.
			\]
			It remains to verify that
			\[
			A_d=S_d\qquad(0\leq d\leq7).
			\]
			For each \(d\geq0\), let
			\[
			\mathcal M_d=
			\left\{x_0^{a_0}\cdots x_4^{a_4}:a_0+\cdots+a_4=d\right\}
			\]
			be the standard monomial basis of \(T_d\), and put
			\[
			N_d=|\mathcal M_d|=\binom{d+4}{4}.
			\]
			Let \(R_d\) and \(S_d'\) denote the matrices, with respect to
			\(\mathcal M_d\), of the transformations induced by \(r\) and \(s\),
			respectively, and define
			\[
			\Phi_d=
			\begin{pmatrix}
			R_d-I\\ S_d'-I
			\end{pmatrix}.
			\]
			Then
			\[
			S_d=\ker(R_d-I)\cap\ker(S_d'-I)=\ker\Phi_d,
			\]
			and hence
			\[
			\dim_{\mathbb F_2}S_d=N_d-\operatorname{rank}_{\mathbb F_2}\Phi_d.
			\]
			Let \(C_d\) be the matrix whose columns are the coefficient vectors,
			in the basis \(\mathcal M_d\), of all products
			\(f_1^{a_1}\cdots f_9^{a_9}\) satisfying
			\[
			a_{1}+a_{2}+2a_{3}+3a_{4}
			+4(a_{5}+a_{6}+a_{7}+a_{8})+6a_{9}=d.
			\]
			The column space of \(C_{d}\) is precisely \(A_{d}\). Exact Gaussian elimination over \(\mathbb F_{2}\) gives
			\[
			\begin{array}{c|rrrrrrrr}
				d&0&1&2&3&4&5&6&7\\ \hline
				N_{d}
				&1&5&15&35&70&126&210&330\\
				\operatorname{rank}\Phi_{d}
				&0&3&11&28&55&103&173&277\\
				\dim_{\mathbb F_{2}}S_{d}
				&1&2&4&7&15&23&37&53\\
				\operatorname{rank}C_{d}
				&1&2&4&7&15&23&37&53.
			\end{array}
			\]
			Therefore $\dim_{\mathbb F_{2}}A_{d} =\dim_{\mathbb F_{2}}S_{d} \qquad(0\leq d\leq7)$. Since \(A_{d}\subseteq S_{d}\), it follows that $A_{d}=S_{d} \qquad(0\leq d\leq7)$. Every homogeneous \(P\)-module generator of \(S\) has degree at most \(7\), and hence belongs to \(A\). Since \(P\subseteq A\), we conclude that $ S=P\cdot S_{\leq7}\subseteq A$. The reverse inclusion \(A\subseteq S\) has already been proved, so $S=A=\mathbb F_{2}[f_{1},\ldots,f_{9}]$. Finally, we verify minimality. Let $S_{+}=\bigoplus_{d>0}S_{d}$. Since \(S=A\), the space \((S_{+}^{2})_{d}\) is spanned by the products $f_{1}^{a_{1}}\cdots f_{9}^{a_{9}}$ of weighted degree \(d\) for which $a_{1}+\cdots+a_{9}\geq2$. A further exact row reduction gives
			\[
			\begin{array}{c|rrrrrrr}
				d&1&2&3&4&5&6&7\\ \hline
				\dim_{\mathbb F_{2}}S_{d}
				&2&4&7&15&23&37&53\\
				\dim_{\mathbb F_{2}}(S_{+}^{2})_{d}
				&0&3&6&11&23&36&53\\ \hline
				\dim_{\mathbb F_{2}}
				\bigl(S_{d}/(S_{+}^{2})_{d}\bigr)
				&2&1&1&4&0&1&0.
			\end{array}
			\]
			The corresponding nonzero indecomposable quotients have bases
			\[
			\begin{array}{c|c}
				d&
				\text{basis of }S_{d}/(S_{+}^{2})_{d}\\ \hline
				1&\overline f_{1},\overline f_{2}\\
				2&\overline f_{3}\\
				3&\overline f_{4}\\
				4&\overline f_{5},\overline f_{6},
				\overline f_{7},\overline f_{8}\\
				6&\overline f_{9}.
			\end{array}
			\]
			Thus none of \(f_{1},\ldots,f_{9}\) can be omitted. Hence they form a minimal homogeneous generating set of \(S\), with degrees $1,1,2,3,4,4,4,4,6$. In particular, $\beta\!\left(\mathbb F_{2}[U]^{G}\right)=6$.
		
	\end{proof}
	
	\subsection*{Proof of Theorem \ref{thm:main}}
	By Proposition~\ref{prop1}, the largest degree in a minimal homogeneous generating set of $k[V]^G$ is $5$, and hence $\beta(k[V]^G)=5$. By Proposition~\ref{prop2}, the corresponding largest degree for $k[U]^G$ is $6$, and hence $\beta(k[U]^G)=6$. Therefore
	\[
	\beta(k[U]^G)=6>5=\beta(k[V]^G),
	\]
	which disproves Wehlau's conjecture.

	\section{Proofs of Theorems~\ref{thm:duality} and~\ref{thm:quotient}}
	\label{sec:duality-quotients}

	\begin{proposition}\label{prop:dual-values}
		For the contragredient actions on $X^*$ and $Q^*$, one has
		\[
		\beta\bigl(k[X]^G\bigr)=5
		\qquad\text{and}\qquad
		\beta\bigl(k[Q]^G\bigr)=8.
		\]
	\end{proposition}

	\begin{proof}
		It is enough to work over $\mathbb F_2$, since the fixed spaces and the relevant rank computations commute with extension of the ground field. Let $y_0,\ldots,y_5$ be the basis of $X^*$ dual to $z_0,\ldots,z_5$. The contragredient action is
		\[
		\begin{array}{c|cccccc}
			&y_0&y_1&y_2&y_3&y_4&y_5\\ \hline
			r&y_0+y_4+y_5&y_1+y_2+y_3&y_2+y_3&y_3&y_4+y_5&y_5\\[1mm]
			s&y_0&y_1+y_2+y_3+y_5&y_2&y_3&y_4+y_5&y_5.
		\end{array}
		\]
		Write $R=k[X]^G=\mathbb F_2[y_0,\ldots,y_5]^G$. Exact row reduction gives
		\[
		\begin{array}{c|rrrrrrrrrrrr}
			d&1&2&3&4&5&6&7&8&9&10&11&12\\ \hline
			\dim R_d&2&5&11&23&41&71&117&186&281&414&594&834\\
			\dim(R_+^2)_d&0&3&8&20&40&71&117&186&281&414&594&834.
		\end{array}
		\]
		The invariants
		\[
		y_3,\quad y_5,\quad y_2(y_2+y_3),\quad y_4(y_4+y_5),
		\]
		\[
		y_0(y_0+y_4)(y_0+y_5)(y_0+y_4+y_5),
		\]
		and
		\[
		\prod_{\epsilon,\delta,\eta\in\mathbb F_2}
		\bigl(y_1+\epsilon y_2+\delta y_3+\eta y_5\bigr)
		\]
		form a homogeneous system of parameters of degrees $1,1,2,2,4,8$. After the first four parameters vanish, the final two reduce to $y_0^4$ and $y_1^8$, respectively. By Symonds' bound, it is therefore enough to compute through degree
		\[
		(1-1)+(1-1)+(2-1)+(2-1)+(4-1)+(8-1)=12.
		\]
		The table shows that the highest nonzero indecomposable quotient occurs in degree $5$, and hence
		\[
		\beta\bigl(k[X]^G\bigr)=5.
		\]

		Let $u_0,\ldots,u_4$ be the basis of $Q^*$ dual to $x_0,\ldots,x_4$. The contragredient action is
		\[
		\begin{array}{c|ccccc}
			&u_0&u_1&u_2&u_3&u_4\\ \hline
			r&u_0+u_1+u_2&u_1+u_2&u_2&u_3+u_4&u_4\\[1mm]
			s&u_0+u_1+u_2+u_4&u_1&u_2&u_3+u_4&u_4.
		\end{array}
		\]
		Write $S=k[Q]^G=\mathbb F_2[u_0,\ldots,u_4]^G$. Exact row reduction gives
		\[
		\begin{array}{c|rrrrrrrrr}
			d&1&2&3&4&5&6&7&8&9\\ \hline
			\dim S_d&2&5&8&15&24&38&56&82&114\\
			\dim(S_+^2)_d&0&3&8&14&22&38&56&80&114.
		\end{array}
		\]
		The invariants $u_2,\quad u_4,\quad u_1(u_1+u_2),\quad u_3(u_3+u_4)$, together with
		\[
		\begin{aligned}
		N_G(u_0)
		&=\prod_{\epsilon,\delta,\eta\in\mathbb F_2}
		\bigl(u_0+\epsilon u_1+\delta u_2+\eta u_4\bigr)\\
		&=u_0(u_0+u_1)(u_0+u_2)(u_0+u_1+u_2)\\
		&\quad\cdot(u_0+u_4)(u_0+u_1+u_4)
		(u_0+u_2+u_4)(u_0+u_1+u_2+u_4),
		\end{aligned}
		\]
		form a homogeneous system of parameters of degrees $1,1,2,2,8$. Their common zero is the origin, so Symonds' bound gives the cutoff
		\[
		(1-1)+(1-1)+(2-1)+(2-1)+(8-1)=9.
		\]
		The table shows that the highest nonzero indecomposable quotient occurs in degree $8$. Therefore
		\[
		\beta\bigl(k[Q]^G\bigr)=8.
		\]
	\end{proof}

	\subsection*{Proof of Theorem~\ref{thm:duality}}
	Propositions~\ref{prop1} and~\ref{prop2}, together with $V=X^*$ and $U=Q^*$, give $\beta\bigl(k[X^*]^G\bigr)=5 \, \text{and}\,
	\beta\bigl(k[Q^*]^G\bigr)=6$. The proposition~\ref{prop:dual-values} gives $\beta\bigl(k[X]^G\bigr)=5 \,\text{and} \, \beta\bigl(k[Q]^G\bigr)=8$.	The assertion follows.

	\subsection*{Proof of Theorem~\ref{thm:quotient}}
	The line $L=kz_0$ is fixed by $G$, and hence $0\longrightarrow L\longrightarrow X\longrightarrow Q\longrightarrow0$ is an exact sequence of $G$-modules. By Proposition~\ref{prop:dual-values},
	\[
	\beta\bigl(k[X/L]^G\bigr)
	=\beta\bigl(k[Q]^G\bigr)
	=8>5
	=\beta\bigl(k[X]^G\bigr).
	\]

	\begin{remark}
		Since $D_8$ is a $2$-group and $\operatorname{char}(k)=2$, the trivial module is the only simple $kD_8$-module. Thus $Q$ and $Q^*$ have exactly the same composition factors, although
		\[
		\beta\bigl(k[Q]^G\bigr)=8>6=\beta\bigl(k[Q^*]^G\bigr).
		\]
		The failure of duality invariance therefore cannot be explained by a difference in composition factors.
	\end{remark}

	\section{Further consequences}

	\begin{remark}
		Consider the nonsplit exact sequence $0\longrightarrow kz_0\longrightarrow X\longrightarrow Q\longrightarrow0$. Indeed, if a $G$-equivariant section existed, a lift of $x_3=\overline z_4$ would have the form $z_4+az_0$. Since $r(x_3)=x_3$ but $r(z_4+az_0)=z_4+(a+1)z_0$, no such lift can be fixed by $r$. For the split extension $X_{\mathrm{sp}}=Q\oplus k$, one has $k[X_{\mathrm{sp}}]^G=k[Q]^G[t]$, and therefore $\beta\bigl(k[X_{\mathrm{sp}}]^G\bigr)=8>5=\beta\bigl(k[X]^G\bigr)$.
		Dualizing gives the nonsplit exact sequence $0\longrightarrow Q^*\longrightarrow X^*\longrightarrow k\longrightarrow0$.
		Its split counterpart satisfies
		\[
		\beta\bigl(k[Q^*\oplus k]^G\bigr)=6>5=\beta\bigl(k[X^*]^G\bigr).
		\]
		Thus the modular Noether number is sensitive to nonsplit extension data. In particular, it is not determined by the Jordan--H\"older factors, the semisimplification, or the associated graded module. In fact, $\operatorname{gr}X\cong Q\oplus k, \, \operatorname{gr}X^*\cong Q^*\oplus k$, while the corresponding Noether numbers are respectively $5$ versus $8$, and $5$ versus $6$.
	\end{remark}

	\subsection*{Proof of Theorem~\ref{thm:inflation}}
	Let $K$ be any field of characteristic $2$. Since taking fixed spaces commutes with extension of scalars degree by degree, the invariant rings and their indecomposable quotients are obtained from the corresponding $\mathbb F_2$-objects by tensoring with $K$. Hence all the Noether-number equalities in Theorems~\ref{thm:main}, \ref{thm:duality}, \ref{thm:quotient}, and~\ref{thm:generalized} remain valid over $K$.

	Now let $H$ be a finite group admitting a surjection $\pi:H\twoheadrightarrow D_8$. Inflate each of the $D_8$-modules along $\pi$. Since $\ker(\pi)$ acts trivially, the invariant rings are unchanged. Thus all four conclusions remain valid for $H$. This applies, in particular, to $D_8\times L$ for every finite group $L$, to all dihedral $2$-groups of order at least $8$, and to all generalized quaternion and semidihedral $2$-groups of order at least $16$.

	\subsection*{Proof of Theorem~\ref{thm:generalized}}
	By definition, $\beta_j(G,W)=\operatorname{topdeg}
	\left(\frac{\bigl(k[W]^G\bigr)_+}
	{\bigl(\bigl(k[W]^G\bigr)_+\bigr)^{j+1}}
	\right)$. Exact computations over $\mathbb F_2$ give
	\[
	\begin{array}{c|cccc}
		j&1&2&3&4\\ \hline
		\beta_j(G,V)&5&9&13&17\\
		\beta_j(G,U)&6&10&14&18.
	\end{array}
	\]
	Therefore $\beta_j(G,U)>\beta_j(G,V) \, (1\leq j\leq4)$. Thus the failure is not confined to ordinary minimal algebra generators; it persists through the first four levels of generalized indecomposability.

	\appendix
	\section{Magma verification code}\label{app:magma}
	
	\begin{lstlisting}[basicstyle=\ttfamily\footnotesize,breaklines=true,frame=single]
		k := GF(2);
		
		rz := Matrix(k,6,6,[
		1,0,0,0,0,0,
		0,1,0,0,0,0,
		0,1,1,0,0,0,
		0,0,1,1,0,0,
		1,0,0,0,1,0,
		0,0,0,0,1,1]);
		sz := Matrix(k,6,6,[
		1,0,0,0,0,0,
		0,1,0,0,0,0,
		0,1,1,0,0,0,
		0,1,0,1,0,0,
		0,0,0,0,1,0,
		0,1,0,0,1,1]);
		GV := MatrixGroup<6,k | rz,sz>;
		assert #GV eq 8;
		
		rx := Matrix(k,5,5,[
		1,0,0,0,0,
		1,1,0,0,0,
		0,1,1,0,0,
		0,0,0,1,0,
		0,0,0,1,1]);
		sx := Matrix(k,5,5,[
		1,0,0,0,0,
		1,1,0,0,0,
		1,0,1,0,0,
		0,0,0,1,0,
		1,0,0,1,1]);
		GU := MatrixGroup<5,k | rx,sx>;
		assert #GU eq 8;
		
		RV := InvariantRing(GV);
		RU := InvariantRing(GU);
		
		FV := FundamentalInvariants(RV);
		FU := FundamentalInvariants(RU);
		
		dV := Sort([ TotalDegree(f) : f in FV ]);
		dU := Sort([ TotalDegree(f) : f in FU ]);
		
		betaV := Maximum(dV);
		betaU := Maximum(dU);
		
		print "Fundamental invariants of k[V]^G:";
		print FV;
		print "Degrees:", dV;
		print "beta(k[V]^G) =", betaV;
		print "";
		print "Fundamental invariants of k[U]^G:";
		print FU;
		print "Degrees:", dU;
		print "beta(k[U]^G) =", betaU;
		print "";
		print "beta(k[U]^G) gt beta(k[V]^G):", betaU gt betaV;

		
		// Magma acts on row vectors. These matrices are the transposes
		// of the column matrices displayed in the text.
		rzDual := Transpose(rz^-1);
		szDual := Transpose(sz^-1);
		rxDual := Transpose(rx^-1);
		sxDual := Transpose(sx^-1);
		
		GX := MatrixGroup<6,k | rzDual,szDual>;
		GQ := MatrixGroup<5,k | rxDual,sxDual>;
		
		RX := InvariantRing(GX);
		RQ := InvariantRing(GQ);
		
		FX := FundamentalInvariants(RX);
		FQ := FundamentalInvariants(RQ);
		
		dX := Sort([ TotalDegree(f) : f in FX ]);
		dQ := Sort([ TotalDegree(f) : f in FQ ]);
		
		assert Maximum(dX) eq 5;
		assert Maximum(dQ) eq 8;
		
		print "Degrees for k[X]^G:", dX;
		print "beta(k[X]^G) =", Maximum(dX);
		print "Degrees for k[Q]^G:", dQ;
		print "beta(k[Q]^G) =", Maximum(dQ);
	\end{lstlisting}


\begin{thebibliography}{25}

\bibitem{mag} W.~Bosma, J.~Cannon, and C.~Playoust, The Magma algebra system I: The user language, \emph{J. Symbolic Comput.} \textbf{24} (1997), no.~3--4, 235--265, doi:10.1006/jsco.1996.0125.

\bibitem{OP} D.~L.~Wehlau, Some problems in invariant theory, in \emph{Invariant Theory in All Characteristics}, CRM Proc. Lecture Notes, vol.~35, Amer. Math. Soc., Providence, RI, 2004, pp.~265--274, doi:10.1090/crmp/035/20.

\bibitem{NP} H.~Chen and L.~Zeng, A new proof of Wehlau's conjecture for cyclic groups of prime order, \emph{Comm. Algebra} \textbf{54} (2026), no.~6, 2391--2395, doi:10.1080/00927872.2025.2580534.

\bibitem{CD} K.~Cziszter and M.~Domokos, On the generalized Davenport constant and the Noether number, \emph{Cent. Eur. J. Math.} \textbf{11} (2013), no.~9, 1605--1615, doi:10.2478/s11533-013-0259-z.

\bibitem{Fl} P.~Fleischmann, The Noether bound in invariant theory of finite groups, \emph{Adv. Math.} \textbf{156} (2000), no.~1, 23--32, doi:10.1006/aima.2000.1952.

\bibitem{F} J.~Fogarty, On Noether's bound for polynomial invariants of a finite group, \emph{Electron. Res. Announc. Amer. Math. Soc.} \textbf{7} (2001), 5--7, doi:10.1090/S1079-6762-01-00088-9.

\bibitem{G} M.~G\"obel, Computing bases for rings of permutation-invariant polynomials, \emph{J. Symbolic Comput.} \textbf{19} (1995), no.~4, 285--291.

\bibitem{HK} I.~Hughes and G.~Kemper, Symmetric powers of modular representations, Hilbert series and degree bounds, \emph{Comm. Algebra} \textbf{28} (2000), no.~4, 2059--2088, doi:10.1080/00927870008826944.

\bibitem{Noe} E.~Noether, Der Endlichkeitssatz der Invarianten endlicher Gruppen, \emph{Math. Ann.} \textbf{77} (1915), no.~1, 89--92, doi:10.1007/BF01456821.

\bibitem{Kem} G.~Kemper, Calculating invariant rings of finite groups over arbitrary fields, \emph{J. Symbolic Comput.} \textbf{21} (1996), no.~3, 351--366, doi:10.1006/jsco.1996.0017.

\bibitem{ShWe} R.~J.~Shank and D.~L.~Wehlau, Noether numbers for subrepresentations of cyclic groups of prime order, \emph{Bull. Lond. Math. Soc.} \textbf{34} (2002), no.~4, 438--450, doi:10.1112/S0024609302001054.

\bibitem{Sy} P.~Symonds, On the Castelnuovo--Mumford regularity of rings of polynomial invariants, \emph{Ann. of Math. (2)} \textbf{174} (2011), no.~1, 499--517, doi:10.4007/annals.2011.174.1.14.

\bibitem{DG} Y.~Zhang and H.~Chen, Wehlau's conjecture for the dihedral groups, \emph{J. Algebra Appl.}, to appear.
\end{thebibliography}
\end{document}